\documentclass{amsart}
\usepackage{amsmath}
\usepackage{amssymb}
\usepackage{amsthm}
\usepackage{epsfig}

\newcommand{\amsprimary}[1]{{\footnotesize\noindent AMS 2010 \textit{Mathematics subject
classification:} Primary #1\vspace{1pc}}}
\newcommand{\keywordsnames}[1]{{\footnotesize\noindent\textit{Key words:} #1\vspace{1pc}}}

\newtheorem{theorem}{Theorem}[section]
\newtheorem{theorem*}{Theorem}
\newtheorem{lemma}{Lemma}[section]

\newtheorem{corollary}{Corollary}[section]

\title{The Ricci flow on a cylinder}

\begin{document}
\author{
Jean C. Cortissoz
\and
Alexander Murcia
}
\address{Departamento de Matem\'aticas, Universidad de los Andes, Bogot\'a DC, COLOMBIA.}
\date{}
\begin{abstract}
In this paper we study the Ricci flow on surfaces homeomorphic  to a cylinder (that is, a product
of the circle with a compact interval).
We prove longtime existence results, results on the asymptotic behavior
of the flow, and we report on an interesting phenomenon: convergence
to constant curvature in the normalised flow,under certain assumptions on the initial data, cannot be exponential 
\end{abstract}
\maketitle
{\keywordsnames { Ricci flow; blow-up; convergence.}}

{\amsprimary {54C44; 35K55.}}

\section{Introduction}

There is no need to talk about the importance of the Ricci flow (on closed and noncompact manifolds without boundary) 
in geometric analysis and low dimensional topology. Moreover, besides its numerous applications in geometric analysis,
another of the great merits of studying the Ricci flow is that it gives a way of understanding the
behavior of certain nonlinear parabolic equations using geometric insights. 
The subject of this paper is the study of the following boundary value problem for the Ricci flow on a 
surface with boundary
\begin{equation}
\label{Ricciunnormalised}
\left\{
\begin{array}{l}
\displaystyle\frac{\partial \tilde{g}}{\partial \tilde{t}}=-\tilde{R}\tilde{g} \quad \mbox{in}\quad M\times\left(0,\tilde{T}\right)\\
k_{\tilde{g}}=\gamma \quad \mbox{on}\quad \partial M\times\left(0,\tilde{T}\right)\\
\tilde{g}=g_0 \quad \mbox{in}\quad M,
\end{array}
\right.
\end{equation}
where $\gamma=k_{g_0}$ is the geodesic curvature of the initial metric, and $M$ is 
a surface homeomorphic to a cylinder $\mathbb{S}^1\times\left[-\rho,\rho\right]$, and the corresponding boundary
value problem for
its normalised version. Recall that the normalised flow is obtained from (\ref{Ricciunnormalised}) by the following
procedure.
Let $\phi\left(\tilde{t}\right)$ be a function such that $g=\phi \tilde{g}$ so that the area $A_g\left(M\right)$  of the surface $M$ with respect to the 
rescaled metric $g$ is kept equal to 1 at all times. Then the
time parameter is rescaled by setting
\[
t\left(\tilde{t}\right)=\int_0^{\tilde{t}}\phi\left(\tau\right)\,d\tau.
\]
By this procedure we obtain the following boundary value problem:
\begin{equation}
\label{Riccinormalised}
\left\{
\begin{array}{l}
\dfrac{\partial g}{\partial t}=\left(r-R\right)g \quad \mbox{in}\quad M\times\left(0,T\right)\\
k_g=\dfrac{\gamma}{\sqrt{\phi\left(t\right)}}  \quad \mbox{on}\quad \partial M\times\left(0,\tilde{T}\right)\\
g=g_0 \quad \mbox{in}\quad M,
\end{array}
\right.
\end{equation}
where 
\[
r=\int R\,dA_g,
\]
is the average scalar curvature.
From now on, all quantities referring to the unnormalised flow will have a tilde, whereas those referring to the normalised flow will not.

The existence theory for the unnormalised flow, and hence that of the normalised flow, is well known and we refer to \cite{CortissozMurcia} for a short discussion
on the matter. Hence, 
in this work, we shall be concerned with the asymptotic behavior of both versions of the flow, under certain geometric hypotheses, and what prompted us
to write this note is how curious this behavior seems to be. We can even say that we have been rewarded
by finding out some unexpected (at least to us) results. We refer to some of our results as
 unexpected, because we were 
guided by the following principle: the behavior of the flow in surfaces 
with boundary must parallel that of the flow in closed surfaces (and the results of Brendle in \cite{Brendle} seems to support it).
This principle turned out to be wrong.

Indeed, under the assumption that $R_{g_0}>0$ and $k_{g_0}\leq 0$, we shall prove 
 long time existence results for both the normalised and unnormalised flow, but we will show
that whereas in the normalised flow the curvature remains uniformly bounded 
(at least in a set of examples we consider, see Section \ref{moresymmetries}), in the unnormalised
flow it blows up in infinite time; 
this already marks a difference between the case of a cylinder and the corresponding case of closed surface
of zero Euler characteristic, where both the normalised and unnormalised flow are the same. This, of course,
has to do with the fact that $R_{g_0}>0$ is not a possibility in the case of closed surfaces of Euler charachteristic zero.

To continue with the interesting and surprising behavior
we encountered, we show that even though we expect the curvature to converge towards zero, it
does not do it exponentially, contrary to what does happen in the case of closed surfaces. The authors
think both result are interesting enough to merit reporting them; besides
the proofs are elementary, and this may be of interest to the non specialist, and
to those also interested on the behavior of nonlinear parabolic equations.

After trying to entice the reader with the comments made above, let us be more explicit
regarding the results we shall prove in this paper. We begin with our
longtime existence result.

\begin{theorem*}
Assume that the initial metric has the form
\[
g_0=dr^2+f\left(r\right)^2d\theta^2,
\]
and that $R_{g_0}\geq 0$ and $k_{g_0}\leq 0$. Then the normalised and the unnormalised flow exist for all time.
\end{theorem*}
We show the previous result for the normalised flow for initial data of the form $g_0=dr^2+f\left(r\right)^2d\theta^2$
 in Section \ref{longtimeexistence}, and in Section \ref{asymptoticbehavior} we show that whenever the normalised
flow exists for all time so does the unnormalised flow (no symmetry assumptions are required to prove this claim).
This result should be compared with Proposition 2.3 in \cite{CortissozMurcia}, where it is shown, in the case of a disk, 
that under the conditions $R_{g_0}>0$ and $k_{g_0}\leq 0$, the unnormalised flow becomes singular in finite time.

Our main result regarding asymptotic behavior, which does not require any symmetry hypothesis, is the following
theorem, whose proof is also given in Section \ref{asymptoticbehavior}.
\begin{theorem*}
The total scalar curvature for the normalised flow, under the assumption that $k_{g_0}\leq 0$, satisfies
\[
\int_M R\,dA_g\leq \frac{1}{\log\left(1+t\right)}.
\]
\end{theorem*}
So we should expect convergence of the curvature towards zero, and indeed we can prove so assuming more symmetries
(see Section \ref{moresymmetries}).
However, the behavior of the unnormalised flow is quite different: even though the total curvature also goes to zero,
the curvature is blowing up (see Section \ref{asymptoticbehavior}).
\begin{theorem*}
If $k_{g_0}<0$ and is locally constant, and the length of the boundary in the normalised flow remains bounded away from 0, 
then there is constant $c_1>0$ such that $\int_{M}\tilde{R}^{2}d A \geq c_1$ and as a consequence,
there is a constant $c_2>0$ such that
\begin{equation*}
\tilde{R}_{\max}(\tilde{t}) \geq  c_2 \tilde{t}.
\end{equation*}

\end{theorem*}
A natural question to ask is whether there are examples where the length of the boundary remains bounded away from 0,
and the answer is yes, and indeed we expect that this is always the case. Theorem 3 is proved in Section \ref{asymptoticbehavior}.

Finally, as we said above the convergence towards zero curvature
in this
case cannot be exponential, in contrast to the case of closed surfaces and of surfaces with totally geodesic boundary.
In fact, we prove the following theorem.
\begin{theorem*}
Assume that the initial data has the form 
\[
g_0=dr^2+f\left(\theta\right)^2d\theta^2,
\]
satisfies $k_g<0$ in one of the boundary components and is locally constant,
and that $R_{\min}\left(t\right)$ is attained in both components of $\partial M$. Then, there is a constant $c_2>0$ such that
for the normalised flow holds that
\begin{equation*}
R_{\max}(t) \geq \dfrac{2}{t+c_2}.
\end{equation*}
\end{theorem*}
Notice that if $k_g=0$, the convergence is indeed exponential, but once we take $k_g\neq 0$ in one of the boundary components
(and for both $k_g\leq 0$), 
it is not so anymore. This result is proved in Section \ref{asymptoticbehavior}.

\section{Evolution equations and some technical geometric lemmas}

In this section, we collect some results to be used in the proofs of our main results.

\label{technicalstuff}
\subsection{Evolution equations}
The following evolution equation for the normalised flow is well known (a similar formula for the unnormalised 
flow holds, see Proposition 2.1 in \cite{CortissozMurcia}).
\begin{lemma}
\label{evolutioneqn}
In the normalised flow, the curvature satisfies the boundary value problem
\begin{equation}
\left\{
\begin{array}{l}
\dfrac{\partial R}{\partial t}=\Delta R + R\left(R-r\right) \quad\mbox{in}\quad M\times\left(0,T\right)\\
\dfrac{\partial R}{\partial \eta}=k_g R \quad\mbox{on}\quad \partial M\times\left(0,T\right).
\end{array}
\right.
\end{equation}
\end{lemma}
\begin{proof}
The equation in the interior of $M$ is well known
(see \cite{Hamilton}). The formula for the normal derivative requires some clarification: it comes from the 
formula for the normal derivative in the case of the unnormalised flow (Proposition 2.1 in \cite{CortissozMurcia}) 
by noticing that this identity is scaling invariant.
\end{proof}

{\bf Remark. }
Lemma \ref{evolutioneqn} has as an immediate consequence that nonnegative scalar curvature is preserved and we leave the
proof to the reader. This fact
shall be used throughout the paper.

\subsection{Geometrical lemmas}

We define a parallel of latitude $s$ of the barrel as a curve $\mathbb{S}^1\times \left\{s\right\}$. 
\begin{lemma}
\label{parallel}
Consider a metric of the form
\[
ds^2=d\sigma^2+f\left(\theta,\sigma\right)^2d\theta^2
\]
in $M=\mathbb{S}^1\times\left[-\rho,\rho\right]$. Assume that there is $\alpha>0$ such that $R\geq -\alpha$,  $\left|k_g\right|\leq C$, and let $L_s$
be the length of the parallel of latitude $s$. Then for any other parallel (including of course any boundary component), we have an estimate
\[
L_se^{-2\rho\left(\alpha \rho + C\right)} \leq L_q\leq L_s e^{2\rho\left(\alpha \rho + C\right)}.
\]
\end{lemma}
\begin{proof}
We proof the upper bound first. First of all, let
\[
\varphi=\frac{f_{\sigma}\left(\sigma,\theta\right)}{f\left(\sigma,\theta\right)}=-k\left(\sigma,\theta\right),
\]
where $k\left(\sigma,\theta\right)$ is the geodesic curvature of the parallel of latitude $\sigma$ at
the point $\left(\sigma,\theta\right)$.
On the other hand
\[
\varphi_{\sigma}=-\dfrac{R}{2}-\varphi^2\leq \dfrac{\alpha}{2},
\]
so 
\[
\varphi\left(\sigma,\theta\right)\leq C+\alpha \rho.
\]
Hence we have
\[
f\left(q,\theta\right)=f\left(s,\theta\right)e^{\int_{s}^r \varphi\left(\sigma,\theta\right)\,d\sigma}\leq 
f\left(s,\theta\right)e^{2\rho \left(\alpha\rho +C\right)},
\]
so we obtain the upper bound by integration. The lower bound follows from the fact that in the 
previous argument, $s$ and $q$ are arbitrary.
\end{proof}

\begin{lemma}
\label{area1}
Consider a metric of the form
\[
ds^2=d\sigma^2+f\left(\theta,\sigma\right)^2d\theta^2
\]
in $M=\mathbb{S}^1\times\left[-\rho,\rho\right]$. Assume that there is $\alpha>0$ such that $R\geq -\alpha$,  $\left|k_g\right|\leq C$, let $L$ the minimum 
length of length of the boundary components. Then
\[
2\rho Le^{-2\rho\left(\alpha \rho + C\right)}\leq A\left(M\right)\leq 2\rho Le^{2\rho\left(\alpha \rho + C\right)}.
\]
\end{lemma}
\begin{proof}
Our point of departure is the following inequality, obtained in the proof of
the previous lemma, for $\lambda=\pm \rho$ and $-\rho\leq \sigma\leq \rho$ arbitrary,
\[
f\left(\sigma,\theta\right)=f\left(\lambda,\theta\right)e^{\int_{\lambda}^r \varphi\left(\sigma,\theta\right)\,d\sigma}\leq 
f\left(\lambda,\theta\right)e^{2\rho \left(2\alpha\rho +C\right)},
\]
and from this inequality it follows that
\begin{eqnarray*}
A\left(M\right)&=& \int_0^{2\pi}\int_{-\rho}^{\rho}f\left(\sigma,\theta\right)\,d\sigma d\theta\\
&\leq& \int_0^{2\pi}\int_{-\rho}^{\rho} f\left(\lambda,\theta\right)e^{2\rho \left(2\alpha\rho+C\right)}\,d\sigma d\theta\\
&=& 2\rho L e^{2\rho\left(2\alpha\rho+ C\right)}.
\end{eqnarray*}
The lower bound is obtained in a similar fashion.
\end{proof}
 
{\bf Remark.} This lemma has as a corollary that if we keep the geodesic curvatures of the the boundary components and the diameter
of the surface bounded, then if the length of one boundary component goes to zero, so does the area of the surface.
The results of this section will be used to show our longtime existence result for the normalised flow (see next section).

\section{Existence for all time}
\label{longtimeexistence}

To show that the solution to the normalisation of 
(\ref{Ricciunnormalised}) exists for all time, we follow closely the ideas in \cite{Cortissoz} and
correct some innacuracies found in there. 
In this section we will assume the the initial metric is of the form
\[
g_0=dr^2+f\left(\theta\right)^2d\theta^2.
\]
This form of the metric is preserved by both the normalised and the unnormalised flow.

Our purpose is to show the following theorem, which in turn implies the existence of the normalised flow for all
time (see Section 2.1 in \cite{CortissozMurcia}, beware that the proof given in \cite{Murcia} is not correct).
\begin{theorem}
Assume that the initial data satisfies $R_{g_0}\geq 0$ and $k_{g_0}\leq 0$. Then,
for any $T<\infty$, there exists a constant $C\left(g_0,T\right)$, where $R_0$ is the
scalar curvature of the initial metric, such that on $\left(0,T\right)$, the
scalar curvature of a solution to (\ref{Ricciunnormalised}) satisfies
\[
R\leq C\left(g_0,T\right).
\]
\end{theorem}

To begin with the proof, first we consider Poisson's equation

\begin{equation}
\label{Poisson}
\left\{
\begin{array}{l}
\Delta_{g}f=R-r
\quad\mbox{in} \quad M\\
\dfrac{\partial f}{\partial \eta_{g}}=0,
\qquad \mbox{on} \quad \partial M,
\end{array}
\right.
\end{equation}
where $A$ and $k_{g}$ are the area of the surface at time $t$ and the geodesic curvature of the boundary at time $t$ respectively.
We obtain the  following result.
\begin{theorem}
There exists a function $\psi$ such that
\begin{equation}
\frac{\partial f}{\partial t}=\Delta _{g}f+r f+\psi.
\end{equation}
where $\psi$ satisfies an equation
\[
\left\{
\begin{array}{l}
\Delta \psi =- r' \quad \mbox{in} \quad M\times\left(0,T\right)\\
\displaystyle\frac{\partial \psi}{\partial \eta_g}=-k_gR \quad\mbox{on} \quad\partial M\times \left(0,T\right).
\end{array}
\right.
\]
\end{theorem}
\begin{proof}
The proof is to be found in \cite{Cortissoz}. However, the value of the normal derivative must be corrected by the arguments 
given in the proof of Lemma \ref{evolutioneqn} above.

 \end{proof}
 
 The following lemma will be useful.
 
 \begin{lemma}
 Let $r$ be the average scalar curvature. Then for the normalised flow, and any $T<\infty$ 
 (so that the normalised flow is defined on $\left[0,T\right)$) we have
 \[
 \int_0^T r\left(t\right)\,dt<\infty.
 \]
 \end{lemma}
 \begin{proof}
 Notice that for the unnormalised flow the quantity $\displaystyle \int \tilde{R}\,dA_{\tilde{g}}=-\int k_{\tilde{g}}\,ds_{\tilde{g}}$ 
is nonincreasing under the assumption of nonnegative curvature and $k_{\tilde{g}}\leq 0$; it is also scaling
 invariant, and it corresponds to $r$ in the normalised flow. Hence $r$ is bounded above, and this shows the lemma.
 
 \end{proof}

\begin{lemma}
\label{infiniteintegral}
Let $\overline{R}\left(t\right)$ be the value
of the scalar curvature when restricted to one of the components of the boundary at time $t$. Then we have that
\[
\int_0^T \overline{R}\left(t\right)\,dt<\infty.
\]
\end{lemma}
\begin{proof}
First notice that the conformal factor at any point is given by
\[
u\left(x,t\right)=\exp\left(\int_0^t r-R\left(x,t\right)\,dt\right),
\]
so if the conclusion is false, by Lemma \ref{area1} we must have that the area of the surface approaches 0.
Indeed, the diameter remains bounded, since $R\geq 0$, and by the previous lemma $\int_0^t r\,d\tau<\infty$ for 
any finite $t$.
This would contradict the fact that the normalised flow keeps the area constant.
\end{proof}

 We employ now the notation
 \[
 \left\|\nabla\psi\left(t\right)\right\|_{\infty}=\max_{p\in M}\left\|\nabla\psi\left(p,t\right)\right\|,
 \]
where $\left\|\nabla\psi\left(p,t\right)\right\|$ represents the norm of $\nabla\psi\left(p,t\right)$ with respect to $g\left(t\right)$.
 Then we have that
 \begin{lemma}
\label{lemmapsi}
 We have that for any $T<\infty$
 \[
 \int_0^T \left\|\nabla\psi\left(t\right)\right\|_{\infty}\,dt<\infty.
 \]
 \end{lemma}

For the proof we will use the following elementary result whose proof we leave to the interested reader.
 \begin{lemma}
\label{calculuslemma}
 Assume that $\beta_1$ and $\beta_2$ are continuous positive functions such that
 \[
 \int_0^T \max\left\{\beta_1,\beta_2\right\}\,d\tau=\infty,
 \]
 then there is an $j=1,2$ for which
 \[
 \int_0^T \beta_j\left(\tau\right)\,d\tau=\infty.
 \]
 \end{lemma}
 
 \begin{proof}[Proof of Lemma \ref{lemmapsi}]
 Notice that by Bochner's formula $\left|\nabla \psi\right|^2$ is subharmonic, i.e.,
 \[
 \Delta \left|\nabla \psi\right|^2\geq 0.
 \]
 Hence, by our symmetry assumptions, which imply the same symmetry for $\psi$,
if we let $\beta_j$, $j=1,2$, to be $-k_g R$ in each of the components of the boundary,
then we have
 \[
 \left\|\nabla\psi\left(t\right)\right\|_{\infty}=\max\left\{\beta_1,\beta_2\right\}
 \]
 But if 
 \[
 \int_0^T \left\|\nabla\psi\left(t\right)\right\|_{\infty}\,dt=\infty,
 \]
 since $\left|k_g\right|$ remains uniformly bounded away from 0 in at least one boundary component
over any finite interval of time 
(notice that if $k_g\equiv 0$ there would be nothing to prove, as this case
is already considered in \cite{Brendle}), we must also have
 \[
 \int_0^T \hat{R}\left(t\right)\,dt=\infty,
 \]
where $\hat{R}$ is the maximum of the scalar curvature when restricted to the boundary, but this contradicts Lemma
\ref{infiniteintegral}.

 \end{proof}

Now we 
let 
\[
h:=\Delta_{g}f+\left| \nabla f \right|^{2},
\]
and compute an evolution equation.
\begin{theorem}
$h$ satisfies an evolution equation
\[
\left\{
\begin{array}{l}
\displaystyle\frac{\partial h}{\partial t}=\Delta h -2\left|M_{ij}\right|^2+rh-r'-2\left<\nabla \psi,\nabla f\right>\quad\mbox{in}\quad M\times\left(0,T\right)\\
\displaystyle\frac{\partial h}{\partial \eta}=k_gR \quad \mbox{on}\quad \partial M\times\left(0,T\right).
\end{array}
\right.
\]
\end{theorem}

The proof of this theorem can be found in \cite{Cortissoz}.
From the previous result we find that $h$ satisfies the differential inequality
\[
\frac{\partial h}{\partial t}\leq \Delta h+\left(r+2\left\|\nabla\psi\right\|_{\infty}\right)h-r'+2\left(r+\frac{1}{4}\right)\left\|\nabla\psi\right\|_{\infty}.
\]
Let 
\[
c\left(t\right)=\int_0^t\left(r+2\left\|\nabla\psi\right\|_{\infty}\right)\,dt,
\]
and define
\[
u=\exp\left(-c\left(t\right)\right)h.
\]
Recall that for any finite $T$, both $\displaystyle \int_0^Tr\,dt$ and $\displaystyle \int_0^T \left\|\nabla \psi\right\|_{\infty}\,dt$ are  finite. 
Then $u$ satisfies the differential inequality
\[
\frac{\partial u}{\partial t}\leq \Delta u-\exp\left(-c\left(t\right)\right)\left[r'-2\left\|\nabla\psi\right\|_{\infty}\left(r+\frac{1}{4}\right)\right].
\]
Finally let
\[
v=u+\int_0^t\exp\left(-c\left(t\right)\right)\left[r'-2\left\|\nabla\psi\right\|_{\infty}\left(r+\frac{1}{4}\right)\right]\,d\tau;
\]
then, since $k_g\leq 0$ and $R\geq 0$, $v$ is easily seen to satisfy
\[
\left\{
\begin{array}{l}
\dfrac{\partial v}{\partial t}\leq \Delta v \quad \mbox{in} \quad M\times\left(0,T\right)\\
\dfrac{\partial v}{\partial \eta}\leq 0 \quad \mbox{on}\quad \partial M \times\left(0,T\right).
\end{array}
\right.
\]
By the maximum principle, $v$ is uniformly bounded on $\left(0,T\right)$, and in consequence
so is $R$.

\section{Asymptotic behaviour}
\label{asymptoticbehavior}

In this section we remove the assumption of any symmetry. The results we shall present, unless otherwise stated,
 are valid as long as the initial 
data has nonnegative curvature. First, we recall a result from \cite{Murcia} (notice that in the statement presented here
the hypothesis of symmetry have been removed). 

\begin{theorem}
In the unnormalised flow, the total curvature satisfies the estimate
\begin{equation}
\label{lobachevski1}\int_ {M}\tilde{R}dA_{\tilde{g}} \leq \dfrac{c}{\tilde{t}}.
\end{equation}
\end{theorem}

\begin{proof}
As said before, the quantities with a tilde refer to the unnormalised flow (for instance $\tilde{A}$ is
the area of the surface with respect to the metric $\tilde{g}$). We start then by calculating as follows:
\begin{eqnarray*}
\label{lobachevski}\left( \int_ {M}\tilde{R}dA_{\tilde{g}} \right)_{\tilde{t}}=\int_{\partial M} k_{\tilde{g}}\tilde{R} ds_{\tilde{g}} &=&
 \left( \int_{\partial M} k_{\tilde{g}} ds_{\tilde{g}}\right) \dfrac{\int_{\partial M} k_{\tilde{g}}\tilde{R} ds_{\tilde{g}} }{\int_{\partial M} k_{\tilde{g}} ds_{\tilde{g}} },
\end{eqnarray*}
writing
\[
r_{\partial}=\dfrac{\int_{\partial M} k_{\tilde{g}}\tilde{R} ds_{\tilde{g}} }{\int_{\partial M} k_{\tilde{g}} ds_{\tilde{g}} },
\]
by the Gauss-Bonnet theorem, we obtain
\begin{eqnarray}
\left( \int_ {M}\tilde{R}dA_{\tilde{g}} \right)_{\tilde{t}}&=& \left(2\int_{\partial M} k_{\tilde{g}} ds_{\tilde{g}}\right) r_{\partial}
=- \left( \int_ {M}\tilde{R}dA_{\tilde{g}} \right) r_{\partial}. \notag
\end{eqnarray}

Using that $\tilde{A}'\left(\tilde{t}\right)=-\displaystyle \int_M \tilde{R}\,dA_{\tilde{g}}$, and
integrating the previous identity, we obtain for some constant $c_{1}>0$ independent of time $\tilde{t}$,

\begin{equation*}
-\tilde{A}'(\tilde{t})= \int_ {M}\tilde{R}dA_{\tilde{g}}  = c_{1}e^{-\int^{\tilde{t}}_{0} r_{\partial}(\sigma) d \sigma}.
\end{equation*}

As we can write  $-\tilde{A}'(\tilde{t})h(\tilde{t})=c_{1}$, where $h(\tilde{t})=e^{\int^{\tilde{t}}_{0} r_{\partial}(\sigma) d \sigma}$, 
we proceed as in \cite{Murcia}, to obtain
\begin{eqnarray*}
c_{1}\tilde{t}=-\int^{\tilde{t}}_{0}\tilde{A}'(\sigma)h(\sigma) d \sigma &=&
-\tilde{A}(\tilde{t})h(\tilde{t})+\tilde{A}(0)+\int^{\tilde{t}}_{0}\tilde{A}(\sigma)h'(\sigma) d \sigma\\
&\leq&-\tilde{A}(\tilde{t})h(\tilde{t})+\tilde{A}(0)+\tilde{A}(0)\int^{\tilde{t}}_{0}h'(\sigma) d \sigma \\
&=&-\tilde{A}(\tilde{t})h(\tilde{t})+\tilde{A}(0)+\tilde{A}(0)h(\tilde{t})-\tilde{A}(0) \\
&=&h(\tilde{t})\left( \tilde{A}(0)-\tilde{A}(\tilde{t})\right) \\
&\leq&h(\tilde{t})\tilde{A}(0).
\end{eqnarray*}
As a consequence, we arrive at estimate (\ref{lobachevski1}).

\end{proof}

The previous theorem has the following consequence for the behavior of the total curvature in the
normalised flow.

\begin{theorem}
\label{thmnormalised}
Under normalised Ricci flow, the total scalar curvature satisfies
\begin{equation}
\int_{M}R \,d A_{g} \leq \dfrac{c}{\log\left(t+1\right)},
\end{equation}
for some positive constant $c$ independent of time $t$.
\end{theorem}
\begin{proof}
From the previous lemma we have that in the unnormalised flow the total scalar
curvature satisfies the estimate
\[
\int_{M}\tilde{R} d A_{\tilde{g}}\leq \dfrac{c}{\tilde{t}},
\]
for some positive constant $c$ independent of time $\tilde{t}$.
Since under the Ricci flow the area $\tilde{A}\left(\tilde{t}\right):=A_{\tilde{g}(\tilde{t})}$ of the surface is decreasing and its derivative satisfies  
$\tilde{A}'(\tilde{t})=-\int_{M}\tilde{R}\,dA_{\tilde{g}(\tilde{t})}$,  we can assume, without loss of generality and to simplify the
estimates below,  that $-\tilde{A}'(\tilde{t})\leq 1$. So,  
\begin{equation*}
-\dfrac{\tilde{A}'(\tilde{t})}{\tilde{A}(\tilde{t})}\leq \dfrac{1}{\tilde{A}(\tilde{t})}.
\end{equation*}
Intregrating with respect to $\tilde{t}$ we obtain
\begin{equation*}
\log \left(\dfrac{\tilde{A}(0)}{\tilde{A}(\tilde{t})}\right) \leq \int^{\tilde{t}}_{0} \dfrac{1}{\tilde{A}(\sigma)} d \sigma=t,
\end{equation*}
which implies
\begin{equation*}
\dfrac{\tilde{A}(0)}{\tilde{A}(\tilde{t})} \leq e^{t}.
\end{equation*}
Now, we integrate the previous expresion with respect to time $\tilde{t}$. This yields

\begin{eqnarray*}
\tilde{A}\left(0\right)t&=&
\tilde{A}(0)\int_0^{\tilde{t}} \frac{1}{\tilde{A}(\tilde{t})}\,d\tilde{t}\\
&\leq& \int^{\tilde{t}}_{0}e^{t}d \tilde{t}
= \int^{\tilde{t}}_{0}e^{t}\tilde{A}(\tilde{t})d t \\
& \leq &\tilde{A}(0) \int^{\tilde{t}}_{0}e^{t}d t=A_{0}\left(e^{\tilde{t}}-1\right).
\end{eqnarray*}
Hence we get that
\begin{equation}
\label{surprisinginequality}
\dfrac{1}{\tilde{t}}\leq \dfrac{1}{\log\left(t+1 \right)}.
\end{equation}
Since, the total scalar curvature is scaling-invariant, i.e $\int_{M}RdA_{g(t)}=\int_{M}\tilde{R}dA_{\tilde{g}(\tilde{t})}$, then the result follows.
\end{proof}

The previous result and its proof have two interesting consequences. The first one is that the unnormalised flow must exist for all time: otherwise inequality
(\ref{surprisinginequality}) would not be valid. On the other hand,
we must have that the minimum of the scalar curvature $R_{min}\left(t \right)\rightarrow 0$ as $t \rightarrow \infty$.
We state the first of these facts in the following corollary.
\begin{corollary}
\label{longtimeunnormalised}
If the normalised flow (\ref{Riccinormalised}) exists for all time, then the unnormalised Ricci flow (\ref{Ricciunnormalised}) also exists for all time.
As a consequence, if the initial data is of the form
\[
g_0=dr^2+f\left(\theta\right)^2d\theta^2,
\]
the unnormalised flow exists for all time.
\end{corollary}
The informed reader must compare this result with the case of closed surfaces of Euler characteristic 0: both the
normalised and unnormalised flow coincide, so both exist for all time. 

As a consequence of Corollary \ref{longtimeunnormalised}, the curvature in the 
unnormalised flow (at least when assuming symmetric initial data, see Section \ref{moresymmetries} below) remains bounded on any finite interval of time; but 
it does not remain uniformly bounded, as the following
result shows. 
\begin{theorem}
\label{thmblow}
If the length of the boundary in the normalised flow
remains bounded below, and $k_g<0$ and is locally constant, then there is constant $c_1>0$ such that $\int_{M}\tilde{R}^{2}d A_{\tilde{g}} \geq c_1$ and as a consequence,
there is a constant $c_2>0$ such that
\begin{equation*}
\tilde{R}_{\max}(\tilde{t}) \geq  c_2 \tilde{t}.
\end{equation*}
\end{theorem}

\begin{proof}
If we assume for the normalised Ricci flow that the length of the boundary components are bounded away from 0, then there is a constant $C>0$ such that
for all $t>0$
\begin{equation*}
\dfrac{1}{\sqrt{A_{\tilde{g}}}}l_{\tilde{g}}\left(\partial M\right)=l_{g}\left(\partial M\right)\geq C.
\end{equation*}
This inequality implies
\begin{equation*}
-\dfrac{\tilde{A}'(\tilde{t})}{\sqrt{\tilde{A}}}=\dfrac{\int_{M}\tilde{R}dA_{\tilde{g}}}{\sqrt{\tilde{A}}}=-cl_{g}\left(\partial M\right)\geq C,
\end{equation*}
where $c<0$ is the minimum of the geodesic curvature of the boundary.
On the other hand, by the Cauchy-Schwarz inequality we have
\begin{equation*}
-\tilde{A}'(\tilde{t})=\int_{M}\tilde{R}d A_{\tilde{g}} \leq \sqrt{\int_{M}\tilde{R}^{2}d A_{\tilde{g}}} \sqrt{\tilde{A}(\tilde{t})}.
\end{equation*}
Therefore, we obtain the inequality

\begin{equation*}
\sqrt{\int_{M}\tilde{R}^{2}d A_{\tilde{g}}} \geq \dfrac{-\tilde{A}'(\tilde{t})}{\sqrt{\tilde{A}}} \geq C.
\end{equation*}

To show that $\tilde{R}_{\max}(t)$ blows-up in infinite time, we use this lower bound as follows 

\begin{equation*}
 C^2 \leq  \int_{M}\tilde{R}^{2}d A_{\tilde{g}} \leq \tilde{R}_{\max}(t) \int_{M}\tilde{R} d A_{\tilde{g}} \leq \tilde{R}_{\max}(t) \dfrac{c}{\tilde{t}}.
\end{equation*}

\end{proof}

One result to be expected in geometric flows is that whenever there is convergence, this
convergence is exponential. In our case this is not true;
even though the curvature seems to be approaching 0 (at least it does it in an $L^1$ sense, and in some cases, as shown below, uniformly), 
it does not do it too fast, and
by this we mean exponentially fast,
as the following estimate shows.
\begin{theorem}
Assume that the initial data is of the form 
\[
g_0=dr^2+f\left(\theta\right)^2d\theta^2,
\]
satisfies $k_g<0$ in one of the boundary components and is locally constant,
and that $R_{\min}\left(t\right)$ is attained in both components of $\partial M$. Then, there is a constant $c_2>0$ such that
for the normalised flow holds that
\begin{equation*}
R_{\max}(t) \geq \dfrac{2}{t+c_2}.
\end{equation*}
\end{theorem}
\begin{proof}
We let $k_1$ and $k_2$ be the constant value of the geodesic curvature
in each connected component of the boundary, and let $\tilde{l}_1$ and
$\tilde{l}_2$ be the lengths of each component.
Now notice that, by the Gauss-Bonnet theorem,
\[
-\tilde{A}'=\int \tilde{R}\,d A_{\tilde{g}}=2\int_{\partial M}k_{\tilde{g}}\,d s_{\tilde{g}}=-2k_1\tilde{l}_1-2k_2\tilde{l}_2.
\]
On the other hand
\begin{eqnarray*}
\tilde{A}''&=&-\int \Delta \tilde{R}\,d A_{\tilde{g}}\\
&=&-\int k_{\tilde{g}}\tilde{R}\,d s_{\tilde{g}}
= -\tilde{R}_{\min}\left(t\right)\left(k_1\tilde{l}_1+k_2\tilde{l}_2\right),
\end{eqnarray*}
from which obtain
\[
-\dfrac{\tilde{A}''(\tilde{t})}{\tilde{A}'(\tilde{t})}= \frac{1}{2}\tilde{R}_{\min}.
\]
But
\[
-\tilde{A}'\geq \tilde{R}_{\min}\tilde{A},
\]
and hence
\begin{equation*}
-\dfrac{\tilde{A}'(\tilde{t})}{\tilde{A}(\tilde{t})} \geq -\dfrac{2\tilde{A}''(\tilde{t})}{\tilde{A}'(\tilde{t})}.
\end{equation*}
Then
\begin{equation*}
\dfrac{1}{\tilde{A}(\tilde{t})} \geq \dfrac{2\tilde{A}''(\tilde{t})}{(\tilde{A}'(\tilde{t}))^{2}},
\end{equation*}
and we can compute
\begin{equation*}
t=\int^{\tilde{t}}_{0}\dfrac{1}{\tilde{A}(\sigma)}d\sigma \geq 2\int^{\tilde{t}}_{0}\dfrac{\tilde{A}''(\sigma)}{(\tilde{A}'(\sigma))^{2}}
\,d\sigma=\frac{2}{\tilde{A}'\left(0\right)}-\dfrac{2}{\tilde{A}'(\tilde{t})}.
\end{equation*}
The assumption $R_{g_0}>0$, implies that for an $\epsilon>0$, $-\epsilon=\frac{2}{\tilde{A}'\left(0\right)}$, so
we have
\begin{equation*}
t+\epsilon \geq -\dfrac{2}{\tilde{A}'(\tilde{t})}.
\end{equation*}
which implies, via the fact that $\int \tilde{R}_{\tilde{g}}\,dA_{\tilde{g}}=\int R_g\,dA_{g}\leq R_{\max}\left(t\right)$, 
\begin{equation*}
R_{\max}(t) \geq -\tilde{A}'(\tilde{t}) \geq \dfrac{2}{t+\epsilon}.
\end{equation*}
\end{proof}
We must point out that examples of initial data so that $R_{\min}\left(t\right)$ is attained in both boundary components of $M$
are easy to construct.

\medskip
{\bf Remark. } Notice that in the proofs of Theorems \ref{thmnormalised} and \ref{thmblow} all that  is required is that $\int R_g\,dA_g$ remains positive on any
finite interval $\left[0,T\right]$, which is true as long as $k_g<0$ in any of the connected components of $\partial M$.

\subsection{Assuming more symmetries: refined asymptotic results.}
\label{moresymmetries}
If we assume 
additional symmetries on the initial metric, we can sharpen our results on the behavior of the curvature. Notice that so far we have been able to
prove that the scalar curvature remains bounded over finite time intervals. 
Let us recall some terminology used so far, we write
\[
M=\mathbb{S}^1\times\left[-\rho,\rho\right],
\] 
and we will call $\mathbb{S}^1\times\left\{0\right\}$ the middle parallel.

Now we shall assume that not only that the initial data is of the form
\[
g_0=dr^2+f\left(\theta\right)^2d\theta^2
\]
but also
that it is symmetric by reflection with respect to the middle parallel. We will
assume as well that the scalar curvature $R_{g_0}$ is decreasing as we move from the
middle parallel towards any of the boundary components. These properties of the
initial data are preserved under the Ricci flow as considered in this article, and we will say
in this case that the scalar curvature is decreasing from the middle. Examples 
where the metric satisfies the properties described above are easy to construct
(see Proposition 3 in \cite{Murcia}). 
The following results shows that
under these additional assumptions we can prove that the curvature is uniformly bounded above and even that it approaches 0 as
$t\rightarrow\infty$.

\begin{theorem} 
\label{subconvergingtozero}
If the scalar curvature of the initial data is decreasing from the middle, then there exists a sequence of times $t_k\rightarrow \infty$ such that
$R\left(t_k\right)\rightarrow 0$.
\end{theorem}
\begin{proof}
Notice that being the minimum located at the boundary, there exists a $c>0$ such that the length of each boundary component is bounded at
least $c$ for all time. Hence, the diameter of the barrel must remain uniformly bounded by the results 
of Section 2. On the other hand the length of the middle parallel
behaves as
\[
L\left(t\right)=L_0e^{\int_0^tr-R_{\max}\,d\tau},
\]
and hence it is decreasing. But the length of the middle parallel must remain bounded away from 0; otherwise, the area of the barrel would go to 0.
Hence, $\int_0^t R_{\max}-r\,dt$ remains uniformly bounded (incidently, notice that
the integrand is positive), so for any $k>0$ there must exists a $t_k$ such that $R_{\max}\left(t_k\right)-r<\frac{1}{k}$. But we have
already shown that $r\rightarrow 0$ as $t\rightarrow \infty$, and this implies that $R\left(t_k\right)\rightarrow 0$.
\end{proof}

Now we prove that the curvature remains uniformly bounded.

\begin{theorem}
\label{uniformlybounded}
The scalar curvature remains uniformly bounded in the normalised flow.
\end{theorem}

\begin{proof}
The maximum of the scalar curvature satisfies
\[
\frac{d}{dt}R_{\max}\leq R_{\max}\left(R_{\max}-r\right), 
\]
and hence, given any $\tau$, we have that
\begin{equation}
\label{maximumbounded}
R_{\max}\left(t\right)\leq R_{\max}\left(\tau\right)e^{\int_{\tau}^t R_{\max}\left(z\right)-r\left(z\right)\,dz},
\end{equation}
and we know that $\int_{\tau}^t R_{\max}\left(z\right)-r\left(z\right)\,dz$ is uniformly bounded, so the result follows.
\end{proof}

From the previous two results we can finally conclude:
\begin{corollary}
In the normalised flow $R\left(t\right)\rightarrow 0$ as $t\rightarrow \infty$.
\end{corollary}
\begin{proof}
Just observe that for any $\epsilon>0$, there is a $t_*$ such that $R_{\max}\left(t_*\right)<\epsilon/A$
(by Theorem \ref{subconvergingtozero}),
where $A$ is a bound on $\int_{0}^{\infty} R_{\max}-r\,dt$, and hence by (\ref{maximumbounded}),
$R_{\max}\left(t\right)<\epsilon$ for $t>t_*$.
\end{proof}

\end{document}